\documentclass[12pt,a4paper]{article}

\usepackage[usenames, dvipsnames]{xcolor}
\usepackage{amsmath}
\usepackage{amsfonts}
\usepackage{amssymb}
\usepackage{amsthm}
\usepackage{array}
\usepackage{hyperref}
\usepackage{url}
\usepackage{graphicx}
\usepackage{listings}
\usepackage{algorithm}

\usepackage{biblatex}
\DeclareMathOperator{\Hom}{Hom}

\DeclareMathOperator{\Compos}{Compos}
\DeclareMathOperator{\Reg}{Reg}

\DeclareMathOperator{\Ind}{Ind}
\DeclareMathOperator{\Res}{Res}
\DeclareMathOperator{\End}{End}
\DeclareMathOperator{\Cl}{Cl}
\DeclareMathOperator{\incl}{incl}
\DeclareMathOperator{\id}{id}
\DeclareMathOperator{\alg}{-alg}

\definecolor{darkWhite}{rgb}{0.94,0.94,0.94}

\newcommand{\aurel}[1]{{\color{green}[Aurel: {#1}]}}

\newcounter{introthm}

\newtheorem{introtheorem}[introthm]{Theorem}

\newtheorem{tm}{Theorem}[section]
\newtheorem{pro}[tm]{Proposition}
\newtheorem{lm}[tm]{Lemma}
\newtheorem{cor}[tm]{Corollary}

\theoremstyle{definition}
\newtheorem{df}[tm]{Definition}
\newtheorem{req}[tm]{Remark}
\newtheorem{ex}[tm]{Example}
\newtheorem{algo}[tm]{Algorithm}
\newtheorem{ft}[tm]{Fact}

\usepackage{tikz}
\usepackage{pgf}
\usetikzlibrary{arrows,automata}
\usetikzlibrary{positioning}
\usetikzlibrary{trees}
\usepackage{comment}

\usepackage{tkz-tab}

\newcommand{\N}{\mathbb{N}}

\newcommand{\Q}{\mathbb{Q}}
\newcommand{\Z}{\mathbb{Z}}
\newcommand{\C}{\mathbb{C}}
\newcommand{\CC}{\mathcal{C}}

\newcommand{\F}{\mathbb{F}}

\renewcommand{\O}{\mathcal{O}}

\newcommand{\HH}{\mathcal{H}}
\newcommand{\J}{\mathcal{J}}
\newcommand{\G}{\mathcal{G}}
\newcommand{\Y}{\mathcal{Y}}
\newcommand{\NN}{\mathcal{N}}

\date{}
\author{Fabrice Etienne}
\title{Computing class groups by induction with generalised norm relations}

\begin{document}

\maketitle
\abstract{We introduce a generalisation of norm relations in the group algebra $\Q[G]$, where $G$ is a finite group. We give some properties of these relations, and use them to obtain relations between the $S$-unit groups of different subfields of the same Galois extension of $\Q$, of Galois group $G$. Then we deduce an algorithm to compute the class groups of some number fields by reducing the problem to fields of lower degree. We compute the class groups of some large number fields.}

\section*{Introduction}

The problem of computing the class group of number fields is a central problem in number theory, with applications for example in cryptography or in class field theory.
Let $K$ be a number field. The most commonly used method to compute the class group or the $S$-units groups of $K$ 
is Buchmann's algorithm \cite{Buchman}. Its complexity grows quickly with the degree~$n$: if we denote by $D(K)$ the discriminant of $K$, the time complexity of this algorithm for fixed~$n$ is in $\O( e^{a \sqrt{ \ln|D(K)|\ln \ln|D(K)|}} )$ where $a$ is a constant,
and the implicit constant of the $\O$ depends on $n$ exponentially; note in addition that the discriminant grows at least exponentially with~$n$. So it would be beneficial to have an inductive method to reduce these problems to similar ones in number fields of smaller degree and discriminant.

Such inductive methods have already been proposed. Suppose $K/F$ is a Galois extension of number fields, of Galois group $G$.
By studying relations between the subgroups of $G$ arising from character theory, we can find corresponding relations between the arithmetic invariants of the intermediate fields. For a subgroup $H< G$, denote by $\Ind_{G/H}(1_H)$ the permutation character of $G$ induced from the trivial representation of~$H$. A Brauer relation \cite{Brauer} is a relation of the form $$\sum_{H<G} a_H \Ind_{G/H}(1_H) = 0$$ with $a_H \in \Z$. Brauer proved that when such a relation exists, there is a corresponding relation between arithmetic invariants of the fields~$K^H$.

In~\cite{norm}, Biasse, Fieker, Hofmann and Page studied another type of relation called norm relation. Given a subgroup $H < G$, we define its norm element to be the formal sum $N_H = \sum_{h \in H}h$ in $\Z[G]$. If $R$ is a commutative ring, $\HH$ a set of subgroups of $G$, a \emph{norm relation} over $R$ with respect to $\HH$ is an equality in $R[G]$ of the form $$1 = \sum_{i = 1}^\ell a_i N_{H_i} b_i$$ with $a_i, b_i \in R[G]$ and $H_i<G$. In their paper, they derive from such a relation an inductive algorithm to compute the class group or the groups of $S$-units of $K$ by reducing the problem to a similar problem on the subfields $K^H$.

Our goal is to generalize the notion of norm relation in order to be able to use the same kind of method to compute the class group of some number fields $K$ even without a Galois extension $K/F$. If $K$ is a number field, let us denote by $\Tilde{K}$ its Galois closure, and let $G$ be its Galois group (i.e. the Galois group of the extension $\Tilde{K}/\Q$). Let $H$ be the subgroup of $G$ such that $K = \Tilde{K}^H$. Let $R$ be a commutative ring, and $J_1, \cdots, J_l$ some subgroups of $G$. A \emph{generalised norm relation} over $R$ with respect to $H$ and the $J_i$ is an equality in $R[G]$ of the form $$ N_H = \sum_{i = 1}^\ell a_i N_{J_i} b_i $$ with $a_i, b_i \in R[G]$. 
Classical norm relation are a special case of generalised norm relation where $H$ is the trivial group.
Given $K$ and subfields~$K_i$ of $\Tilde{K}$, we say that $K$ admits a generalised norm relation with respect to the $K_i$ if~$G$ admits one over $\Q$ with respect to $H$ and $\{J_i\}$, where the $J_i$ are such that $K_i = \Tilde{K}^{J_i}$ for all~$i$.
We impose the condition that the $K_i$ are subfields of the Galois closure~$\Tilde{K}$, but we prove in theorem \ref{tm:galoisclosure} this causes no loss in generality. 
The main result of this article is the following:

\begin{introtheorem} \label{tmA}
There exists a polynomial time algorithm that, on input
\begin{itemize}
    \item a number field~$K$,
    \item a set $S$ of prime numbers,
    \item subfields $K_i$ of the Galois closure $\Tilde{K}$,
    \item for each~$i$, a basis of the $S$-unit group of $K_i$,
\end{itemize}
if $K$ admits a generalised norm relation with respect to the $K_i$, outputs a basis of the $S$-unit group of $K$.
\end{introtheorem}

We will describe such an algorithm (algorithm \ref{algo:polynomial}), as well as another algorithm that is not provably polynomial time, but is often faster than the first one in practice (algorithm \ref{algo:faster})
Using an implementation of such an algorithm in Pari/GP, we compute the class group of some number fields significantly faster than with other methods, including the method using classical norm relations from~\cite{norm}. In particular, in example \ref{ex:C7A5}, we compute the class group of a field of degree $105$ and discriminant $2^{126} \cdot  29^{90} \cdot  67^{42} \simeq 1.7 \cdot 10^{246}$ in about $5$ days (CPU time), whereas without our method, we could not compute it in over $5$ months.

One problem we encounter is that there is no known polynomial time algorithm that, given a number field $K$ defined by an irreducible polynomial over $\Q$, computes its its Galois group $G$. We will provide a way to determine whether relations exists without actually having to compute $G$ or $\Tilde{K}$. 
In order to do that, we will need some properties of Hecke algebras \cite{Yoshida} and compositums. We will prove that a compositum $\CC$ of two number field $K$ and $L$ naturally induces a morphism 
from $K^\times$ to $L^\times$.
This map will be denoted $x \mapsto \CC \cdot x$, and is described by the following theorem (see theorem \ref{tm:action_compo}).

\begin{introtheorem} \label{tmB}
    Let $K, L$ be two number fields, let $x$ be an element of $K^\times$ and $\CC = (C, \iota_K, \iota_L)$ a compositum of $K$ and $L$. Then $\CC \cdot x = N_{C/L}(\iota_K(x))$, where $N_{C/L}$ is the norm of the extension $C/L$.
\end{introtheorem}

We will sometimes refer to this function as the action of the compositum $\CC$ on $K^\times$. A similar result also gives a morphism of additive groups $K\to L$ given by the action of the compositum $\CC$.
We will prove the following characterisation (see theorem \ref{tm:criteria_compositum}).


\begin{introtheorem} \label{tmC}
Let $K,L_1, \cdots , L_\ell$ be number fields. Let $\alpha$ and $\beta_1, \cdots, \beta_l$ be such that $K = \Q(\alpha)$ and $L_i = \Q(\beta_i)$ for all $i$. Then $K$ admits a generalised norm relation with respect to $L_1, \cdots , L_\ell$, if and only there is a relation of the form $$\alpha = \sum_{i = 1}^{\ell} \sum_{C \in Compo(K, L_i)} a_{i, C} C \cdot \beta_i$$
where the coefficients $a_{i,C}$ are in $\Q$.
\end{introtheorem}


To prove theorem \ref{tmA}, we will also need some properties of Mackey functors \cite{Boltje}. In particular proposition \ref{pro:Mackeyfct}, will be useful to find a relation between the $S$-units of the number fields involved in a generalised norm relation. We will prove a bound (theorem \ref{tm:bornec}) that is crucial to prove that our algorithm is indeed polynomial.

We will also compare the method to compute class groups using generalised norm relation with the method described in \cite{norm} that only uses classical norm relations. We will provide an algorithm to find examples where the method using generalised norm relation is more efficient (algorithm \ref{alg:comparaison}).

With a systematic research on every group of cardinal up to 700, it appears that we can find many examples where generalised norm relations are useful.
We think it is an interesting question to classify norm relation (classical ones as well as generalised ones).
We do not have an algorithm that takes a number field and determines whether or not it admits a generalised norm relation in polynomial time, without having to compute the Galois group or the Galois closure. Finding such an algorithm is also an interesting open question.



The article is organised as follows: in section \ref{sec:HeckeCompo}, we discuss properties of Hecke algebras and compositum, then in section \ref{sec:NormRel} we recall the definition of classical norm relations, define the generalisation, and prove some properties. In section \ref{sec:Mackey}, we recall some properties of Mackey functors that will be useful later to prove the complexity of our algorithms. In section \ref{sec:Algo}, we provide algorithms to check whether given fields admit a generalised norm relation, and to compute the class group and the group of $S$-units of number fields. Then in section~\ref{sec:compar}, we compare these new methods with the methods in \cite{norm}. Finally, we give examples in section \ref{sec:Ex}.

\paragraph*{Notations and conventions}
When $R$ is a ring and $M$, $N$ are left $R$-modules, we will denote~$\Hom_R(M,N)$ the group of $R$-module homomorphisms from $M$ to $N$.
In a finite field extension $K/F$, we will denote by $N_{K/F}(x)$ the norm of $x \in K$.

\paragraph*{Acknowledgements} I would like to thank A. Page, who suggested to me to generalise the notion of norm relation, and was here to provide help and advice at every step of the conception of this article. Many thanks also to B. Allombert for his advice regarding the implementation in Pari/GP of the algorithms in section \ref{sec:Algo}, and for his help to compute the $S$-units of some large auxiliary fields.\\
This research was funded by the University of Bordeaux. It was also supported by the CIAO ANR (ANR-19-CE48-008) and the CHARM ANR (ANR-21-CE94-0003), as well as the PEPR Cryptanalyse (ANR-22-PECY-0010). It took place inside the CANARI team (Cryptographic Analysis and Arithmetic) of the Institute of Mathematics of Bordeaux (IMB).\\
Experiments presented in this paper were carried out using the PlaFRIM experimental testbed, supported by Inria, CNRS (LABRI and IMB), Université de Bordeaux, Bordeaux INP and Conseil Régional d’Aquitaine (see https://www.plafrim.fr).

\section{Hecke algebras and compositums} \label{sec:HeckeCompo}

In this section, $K$ is a number field, $G$ is the Galois group of the Galois closure $\tilde{K}$ of $K$, $H$ and $J$ are subgroups of $G$, respectively fixing the subfields $K$ and $L$.
Let $\alpha$ be an element of $\C$ 
such that $K = \Q(\alpha)$ and $f$ be the minimal polynomial of $\alpha$ in $\Q[X]$.
Let $\beta$ be such that $L = \Q(\beta)$ and let $f_L$ be its minimal polynomial.

Our goal in this section is to establish that the set of compositums of $K$ and $L$ acts on the set of fixed points by $H$ of any $G$-module, and to make this action explicit. The action will be described in theorem~\ref{tm:action_compo}.
In this section, $R$ denotes a commutative ring and $V$ an $R[G]$-module.
To establish this theorem, we will first need some useful isomorphisms involving $R[G]$-modules. First let us write two lemmas, describing some well known isomorphisms. One can find the proofs in~\cite{Yoshida}.

\begin{lm}\label{lm:fixed_points}
The map $$\Phi_1\colon \Hom_{R[G]}(R[G/H], V) \rightarrow V^H,
\phi \mapsto \phi(1 \cdot H)$$
is an isomorphism of $R$-modules, and its inverse is

$$\Phi_1^{-1}\colon V^H \rightarrow \Hom_{R[G]}(R[G/H], V), x \mapsto \left\{
    \begin{array}{l}
        \mbox{ The unique morphism } \phi \mbox{ of } \\
        R[G]\mbox{-modules in } V \mbox{ such that } \\
        \phi(1 \cdot H) = x
    \end{array}
\right. .$$
\end{lm}

\begin{lm}
There is an isomorphism of $R$-modules
$$\Phi_2\colon R[H \backslash G /J] \rightarrow R[G/J]^H, \sum_{HgJ \in H \backslash G /J } \alpha_{HgJ}HgJ \mapsto \sum_{gJ \in G/J} \alpha_{HgJ} gJ.$$

Its inverse is $$\Phi_2^{-1}\colon R[G/J]^H \rightarrow R[H \backslash G /J], \sum_{gJ \in G/J} \alpha_{gJ} gJ \mapsto \sum_{HgJ \in H \backslash G /J } \alpha_{gJ} HgJ.$$

\end{lm}

\begin{pro}\label{prop:Heckemorphisms}
There is an isomorphism of $R$-modules
$$\Phi\colon R[H \backslash G /J] \rightarrow \Hom_{R[G]}(R[G/H], R[G/J]) $$ $$ \sum_{HgJ \in H \backslash G / J } \alpha_{HgJ} HgJ \mapsto \left\{
    \begin{array}{l}
        \phi \mbox{ such that } \\
        \phi(1 \cdot H) = \sum_{g \in G/J} \alpha_{HgJ} gJ\\
    \end{array}
\right. .$$
Its inverse is
$$\Phi^{-1}\colon \Hom_{R[G]}(R[G/H], R[G/J]) \mapsto R[H \backslash G /J] $$ $$ \phi \mapsto \left\{
    \begin{array}{l}
        \sum_{HgJ \in H \backslash G /J } \alpha_{gJ} H \delta J \\
        \mbox{where } \phi(1 \cdot H) = \sum_{g \in G/J} \alpha_{gJ} gJ \\
    \end{array}
\right. .$$
\end{pro}

\begin{proof}

We can then obtain the isomorphism $\Phi$ simply by composing $\Phi_1$ and $\Phi_2$ from the two previous lemmas.

\end{proof}

\begin{ft}
By considering both the isomorphism $\Phi$ in proposition \ref{prop:Heckemorphisms} and the isomorphism $\Phi_1$ in lemma \ref{lm:fixed_points}, we deduce that given any $R[G]$-module $V$, for every element $HgJ$ of $R[H \backslash G / J]$ we get a morphism $T_{HgJ}$ of $R$-modules from $V^J$ to $V^H$ given by the following diagram:


\begin{center}
    \begin{tikzpicture}[->,>=stealth',shorten >=2pt,auto,semithick]
        \node (A) {$V^J$};
        \node (B) [right of=A, xshift=8cm] {$V^H$};
        \node (C) [below of=A,, yshift = -4cm] {$\Hom_{R[G]}(R[G/J], V)$};
        \node (D) [below of=B, yshift = -4cm] {$\Hom_{R[G]}(R[G/H], V)$};
        \node (E) [below right of=A, xshift = 0.8cm, yshift = -0.4cm] {$x$};
        \node (F) [below left of=B, xshift = -1.5cm, yshift = -0.4cm] {$\sum_{\underset{HgJ = H\delta J}{\delta \in G/J}} \delta x$};
        \node (G) [above right of=C, xshift = 0.8cm, yshift = 1cm] {$\gamma J \mapsto \gamma x$};
        \node (H) [above left of=D, xshift = -1.5cm, yshift = 1cm] {$\gamma H \mapsto \sum_{\underset{HgJ = H\delta J}{\delta \in G/J}} \gamma \delta x$};
        \path (A) edge [] node [above] {$T_{HgJ}$} (B);
        \path (A) edge [] (C);
        \path (B) edge [] (D);
        \path (C) edge [] node [below] {$\phi_{HgJ}$} (D);
        \path (E) edge [] (G);
        \path (G) edge [] (H);
        \path (H) edge [] (F);
        
    \end{tikzpicture}
\end{center}

Where the expression of $\phi_{HgJ}$ is obtained via the following diagram:

\begin{center}
    \begin{tikzpicture}[->,>=stealth',shorten >=2pt,auto, semithick]
        \node (A) {$R[H \backslash G / J]$};
        \node (B) [right of=A, xshift = 9.5cm] {$R[G/J]^H$};
        \node (C) [below of=B, yshift = -5cm] {$\Hom_{R[G]}(R[G/H], R[G/J])$};
        \node (D) [below right of=A, yshift = -2.5mm, xshift = 3cm] {$\sum_{g \in H \backslash G / J} \alpha_{HgJ}HgJ$};
        \node (E) [below left of=B,  yshift = -2.5mm, xshift = -1.5cm] {$\sum_{g \in G / J} \alpha_{HgJ}gJ$};
        \node (F) [above left of=C, yshift = 2.5cm, xshift = -1.5cm] {$\gamma H \mapsto \sum_{\underset{HgJ = H\gamma J}{g \in G/J}} \alpha_{HgJ}\gamma g J$};
        \path (A) edge [] (B);
        \path (B) edge [] (C);
        \path (A) edge [] (C);
        \path (D) edge [] (E);
        \path (E) edge [] (F);
        \path (D) edge [] (F);        
    \end{tikzpicture}
\end{center}

\end{ft}

\begin{req}
    
This last proposition gives a natural action of~$R[H \backslash G / J]$ that goes from $R[G/H]$ into $R[G/J]$. With the next proposition, we will see that this can also be seen as an action from the set $\Hom(K, \C)$ of complex embeddings of $K$ into the set of complex embeddings of $L$.

\end{req}

\begin{req}
    The Galois group $G$ acts on the set $\Hom(K, \C)$ by $g \cdot \sigma = \sigma \circ g$.

    Recall that $\alpha$ is an element of $\C$ such that $K = \Q(\alpha)$, and $f$ is the minimal polynomial of $\alpha$ over $\Q$.
    Let $\sigma \in \Hom(K, \C)$ and $g \in G$. The embedding $\sigma$ sends $\alpha$ to a complex root of $f$. Then $g \cdot \sigma$ is the element of $\Hom(K, \C)$ that sends $\alpha$ to $g \cdot \sigma(\alpha)$.
\end{req}

\begin{pro} \label{pro:quotient_embedd}
We note $E = \Hom(K, \C)$ the set of embeddings of $K$ in $\C$, and by $\sigma_g$ the embedding that maps $\alpha$ to $g\cdot \alpha$ for all $g \in G$. 

There is an isomorphism $$\Phi\colon G/H \rightarrow E, gH \mapsto \sigma_g.$$
Its inverse is $$\Phi^{-1}\colon S \rightarrow G/H, \tau_g \alpha \mapsto gH.$$
It is independent from the choice of $g$.
\end{pro}

\begin{proof}

 Let $g_1, g_2 \in G$ such that $g_1 H = g_2 H$, and let $h \in H$ such that $g_2 = g_1 h$. Then, $\sigma_{g_2}$ maps $\alpha$ to $g_1 \cdot (h \cdot \alpha) = g_1 \cdot \alpha$. So $\sigma_{g_2} = \sigma_{g_1}$.

\end{proof}

Now from these actions, we will deduce an action of compositums. First let us recall the following definition.

\begin{df}{}
Let $K$ and $L$ be number fields.
A \emph{compositum} of $K$ and $L$ is a triple $(C, \iota_K, \iota_L)$ where $C/\Q$ is a number field, $\iota_K\colon K \rightarrow C$ and $\iota_L\colon L \rightarrow C$ are morphisms of $\Q$-algebras, and where $C$ is generated by $\iota_K(K)$ and $\iota_L(L)$ as a ring.
\end{df}

\begin{ex}

Consider the following diagram, with $\zeta := e^{\frac{2i\pi}{3}}$.

\begin{center}
    \begin{tikzpicture}[->,>=stealth',shorten >=2pt,auto, semithick]
        \node (A) {$C = \Q(\sqrt[3]{2}, \zeta) = \Tilde{K} $};
        \node (B) [left of=A, xshift = -1.5cm, yshift = -2.5 cm] {$K = \Q(\sqrt[3]{2})$};
        \node (C) [right of =A , xshift = 1.5cm, yshift = -2.5cm] {$L = \Q(\zeta)$};
        \node (D) [below of=A, yshift = -4.2cm] {$\Q$};
        \draw [-] (A) to node [above] {$H$} (B); 
        \draw[-] (A) to node [above] {$J$} (C);
        \draw[-] (A) to node [right] {$G$} (D);
        \draw[-] (B) to (D);
        \draw[-] (C) to (D);
    \end{tikzpicture} 
\end{center}

Let $\iota_K\colon K \rightarrow C$ be the inclusion, and $\iota_L\colon L \rightarrow C$ also the inclusion. It is clear that $C$ is generated by $\iota_K(K)$ and $\iota_L(L)$, so $C$ is a compositum of $K$ and $L$.

Here, we have $C = \Tilde{K}$. We will see that up to isomorphism, every compositum of $K$ and $L \subset \Tilde{K}$ is included in $\Tilde{K}$.

Note that if we take $\iota_{K,2}\colon K \rightarrow C$ the inclusion and $\iota_{L, 2}\colon K \rightarrow C_2, \zeta \mapsto \overline{\zeta}$, then $(C, \iota_{K,2}, \iota_{L, 2})$ is another compositum of $K$ and $L$.

\end{ex}

\begin{df}{}

A \emph{morphism of compositums} between two compositums $(C, \iota_K, \iota_L)$ and $(C', \iota_K', \iota_L')$ is a field morphism $f\colon C \rightarrow C'$, such that $\iota_K' = f \circ \iota_K$ and $\iota_L' = f \circ \iota_L$.

\end{df}

The two following lemma give some properties of compositums.


\begin{lm}\label{lm:compoKL}
Up to isomorphism, there is a finite number of compositums of $K$ and $L$, we denote by $\Compos(K,L)$ a set of representatives. There is a bijection between this set and the set of quotients of $K \otimes_\Q L$.
For $f\colon K \otimes_\Q L \rightarrow C$ surjective, the associated compositum is $(C, \iota_K, \iota_L)$ and $\iota_K = f \circ (\id_K \otimes 1)$, $\iota_L = f \circ (\id_L \otimes 1)$.
Every compositum of $K$ and $L$ is isomorphic to a compositum whose underlying field is contained in $\Tilde{K}$.
\end{lm}

\begin{proof}

The second statement is a direct application of the universal property of the tensor product of algebras.

Since $K \otimes_\Q L$ is of finite dimension over $\Q$, the set $\Compos(K,L)$ is finite.

Now let us prove the last statement. Denote $K = \Q[X] / p(X)$, with $p(X) \in \Q[X]$ irreducible. Denote $p(X) = \prod_i p_i(X)$ the decomposition of $p(X)$ into a product of irreducible polynomials in $L[X]$. Then $K \otimes_\Q L = \prod_i L[X] / (p_i(X))$. 
What's more, the polynomials $p_i$ are split in $L[X]$, so they are also split in $\Tilde{K}[X]$. So for every $i$, we have~$L[X] / (p_i(X)) \subset \Tilde{K}$, since~$\Tilde{K}$ contains~$L$ and a splitting field of the $p_i$. Hence the conclusion.
\end{proof}

\begin{lm} \label{lm:compoKL2}
The map $$\Psi\colon \Hom_{\Q \alg}(K,\Tilde{K}) \rightarrow \Compos(K,L), \phi \mapsto (\phi(K).L, \phi, \incl_{L/\Tilde{K}})$$ induces a bijection from $ J \backslash \Hom_{\Q\alg}(K,\Tilde{K})$ to $\Compos(K,L)/\sim$.
(Recall that $J$ is defined to be the subgroup of $G$ that fixes $L$)
\end{lm}

\begin{proof}
Let $\phi \in \Hom_{\Q\alg}(K,\Tilde{K})$. The composition by $g \in J$ induces an isomorphism $(\phi(K).L, \phi, \incl_{L/\Tilde{K}}) \rightarrow (g.\phi(K).L, g.\phi, g.\incl_{L/\Tilde{K}})$.
Since $g \in J$, $g$ fixes $L$, so  $g.\incl_{L/\Tilde{K}} = \incl_{L/\Tilde{K}}$. So the isomorphism induced by $g$ is of the form $\Psi(\phi) \rightarrow \Psi( g \cdot \phi)$. 
Let us check that the map induced by $\Psi$ is injective. Let $\phi, \phi ' \in \Hom_{\Q\alg}(K,\Tilde{K})$ and let $f\colon \phi(K) \cdot L \rightarrow \phi'(K) \cdot L$ an isomorphism of compositums. Then $f \circ \incl_{\Tilde{K}/L} = \incl_{\Tilde{K}/L}$ and $f$ is the identity over $L$, so $f$ can be extended as an isomorphism $g \in J$. Since $f$ is a morphism of compositums, $g \phi = \phi'$, hence $\phi \sim \phi '$.

Let us check it is surjective. By lemma \ref{lm:compoKL}, every compositum in $\Compos(K, L)$ is isomorphic to a compositum where $\iota_L = \incl_{L/ \Tilde{K}}$. Let $\iota_K\colon K \rightarrow \Tilde{K}$ be an embedding, then we can always pick $\phi = \iota_K$.
\end{proof}

\begin{pro}\label{pro:HeckeCompo}
There is a bijection
$$\Phi \colon J\backslash G / H \rightarrow \Compos(K,L), JgH \mapsto (gK \cdot L, l \mapsto gl \cdot 1_{L}, \incl_{L/\tilde{K}}).$$
Its inverse is
$$\Phi^{-1}\colon \Compos(K,L) \rightarrow J\backslash G / H, (C, \iota_K, \iota_L) \mapsto \left\{ \begin{array}{l}
     JgH  \\
     \mbox{with } g \mbox{ such that }\\
     \iota_K(\alpha) = g \cdot \alpha \in \tilde{K}\\
\end{array}
\right. .$$
\end{pro}

\begin{proof}
The proposition derives from the two lemmas \ref{lm:compoKL} and \ref{lm:compoKL2}.
\end{proof}

Thus, using the previous isomorphisms, we obtain an action of $\Compos(K,L)$ on various $R$-modules.

\begin{pro} \label{pro:Compo_morph_Hecke}
The map
$$\Phi\colon \Compos(K, L) \rightarrow \Hom_{R[G]}(R[G/J], R[G/H])$$ $$ (C, \iota_K, \iota_L) \mapsto 
\left\{ \begin{array}{l}
     \phi \mbox{ such that }\\
      \phi(1 \cdot J) = \sum_{\underset{J\gamma H = J g H}{\gamma H \in G/H}} \gamma H \\
      \mbox{with } g \mbox{ such } g\cdot \alpha = \iota_K(\alpha)
\end{array}
\right.$$
is injective.
\end{pro}

\begin{proof}

This is derived from the proposition  \ref{pro:HeckeCompo}, using the isomorphism of proposition \ref{prop:Heckemorphisms}. 
\end{proof}

For all $\alpha'$ a root of $f$, we denote by $\sigma_{\alpha'}$ the embedding of $K$ in $\C$ that sends $\alpha$ to $\alpha'$.
Similarly, denote $\tau_{\beta'}$ the embedding of $L$ in $\C$ that sends $\beta$ to $\beta'$.



\begin{pro}
The map
\begin{align*}
\Phi\colon \Compos(K, L) &\rightarrow &\Hom_{R[G]}(R[\Hom(L, \C)] R[\Hom(K, \C)]), \\
(C, \iota_K, \iota_L) &\mapsto &
\left\{ \begin{array}{l}
     \phi \mbox{ such that }\\
      \phi(\tau_{\beta}) = \sum_{\underset{(C, \iota_K, \iota_L) \sim (C', \sigma, \tau_{\beta})}{\sigma \in \Hom(K, \C)}} \sigma 
\end{array}
\right.
\end{align*}

is injective.
\end{pro}

\begin{proof}
    This is derived from proposition \ref{pro:Compo_morph_Hecke}, using the isomorphism of proposition \ref{pro:quotient_embedd}.
\end{proof}

\begin{req}

Let $(C, \iota_K, \iota_L)$ be a compositum of $K$ and $L$, and let $\phi$ the corresponding element of $\Hom_{R[G]}(R[\Hom(L, \C)], R[\Hom(K, \C)])$. We can obtain a nicer way to write $\phi(\tau_{\beta})$:
$$\phi(\tau_{\beta}) = \sum_{\underset{(C, \iota_K, \iota_L) \sim (C', \sigma, \tau_{\beta})}{\sigma \in \Hom(K, \C)}} \sigma  = 
\sum_{\sigma \in \Hom(K, \C)}
\sum_{E_\sigma} \sigma$$

where $E_\sigma = \{ f \in \Hom(C, \C) | \sigma = f \circ \iota_K \mbox{ and } \tau_{\beta} = f \circ \iota_L \}$.
    
\end{req}

And from that form we can deduce a general expression for $\phi(\tau)$ for every complex embedding $\tau$.

\begin{pro} \label{pro:action_compo_embedd}
Let $(C, \iota_K, \iota_L)$ be a compositum of $K$ and $L$, and let $\phi$ the corresponding element of $\Hom_{R[G]}(R[\Hom(L, \C)], R[\Hom(K, \C)])$. 
For all $\tau \in \Hom(L, \C)$, $$\phi(\tau) = 
\sum_{\sigma \in \Hom(K, \C)}
\sum_{E_\sigma} \sigma$$
where $E_\sigma = \{ f \in \Hom(C, \C) | \sigma = f \circ \iota_K \mbox{ and } \tau_{1_G} = f \circ \iota_L \}$.
\end{pro}

\begin{proof}
Let $\tau = \gamma \cdot \tau_{\beta}$ with $\gamma \in G$. (We can always write $\tau$ in that form, because $g$ acts transitively on the elements of $\Hom(L, \C)$).

Then, $$\phi(\tau) = \gamma \cdot \phi( \tau_{\beta} ) =  
\sum_{\sigma \in \Hom(K, \C)} 
\sum_{E_\sigma} \gamma \cdot \sigma$$
$$ = 
\sum_{\sigma \in \Hom(K, \C)}
\sum_{E_{\gamma \cdot \sigma}} \gamma \cdot \sigma$$ because $\gamma\colon \Hom(K, \C) \rightarrow \Hom(K, \C)$ is a bijection.

$$ = 
\sum_{\sigma \in \Hom(K, \C)}
\sum_{E_\sigma} \sigma$$ because $\gamma\colon \Hom(C, \C) \rightarrow \Hom(C, \C)$ is a bijection.

\end{proof}

Similarly, for every $R[G]$-module $V$, a compositum $C$ of $K$ and $L$ induces a map from $V^H$ to $V^J$. 
(The proof is similar to that of proposition \ref{pro:action_compo_embedd}.)

In the rest of the article, if $x$ is an element of $V^H$, we will denote by $C \cdot x$ the image of $x$ by this map.

\begin{tm} \label{tm:action_compo}
    Let $x$ be an element of $K^\times$ and let $(C, \iota_K, \iota_L)$ be a compositum of $K$ and $L$. Then $C \cdot x = N_{C/L}(\iota_K(x))$.
\end{tm}

\begin{proof}

The bijection described in proposition \ref{pro:HeckeCompo} allows to identify the compositum $(C, \iota_K, \iota_L)$ with an element $J \backslash g / H$ of $J \backslash G / H$.

First, let us prove that the subfield of $\Tilde{K}$ fixed by $H \cap (g J g^{-1}) < G$ is $C$.
The subfield fixed by $g J g^{-1}$ is $g(L) = \iota_L(L)$.
Denote by $\Tilde{C}$ the field fixed by $H \cap (g J g^{-1})$. All elements of $K$ and $\iota_L(L)$ are in $\Tilde{C}$ so $C \subset \Tilde{C}$. What's more, if we denote by $N$ the subgroup of $G$ fixing $C$, then $N$ is included both in $H$ and in $gJg^{-1}$, so it is included in $H \cap gJg^{-1}$. We get $\Tilde{C} \subset C$, so we indeed have $\Tilde{K} ^{H \cap (g J g^{-1})} = C$. 

Now, we know that $C \cdot x = \prod_{\delta \in HgJ/J} \delta x = \prod_{\underset{ HgJ = H\delta J}{\delta \in G/J}} \delta x$ 
We want to make the change of variables $\delta = h g$. For $h, h' \in H$, we have $hgJ = h'gJ$ if and only if there exists $j \in J$ such that $h = h'(gjg^{-1})$, that is to say if and only if $\overline{h} = \overline{h'}$ in $H/(H \cap (g J g^{-1}))$.
This gives $\CC\cdot x = \prod_{h \in H/(H \cap (g J g^{-1}))} h g x$.
Finally, we obtain
$$\CC\cdot x = N_{C/M}(\iota_L(x))$$
as claimed.

\end{proof}

\section{Classical and generalised norm relations} \label{sec:NormRel}

In this section we introduce generalised norm relations, which will be the main type of objects that we will study throughout this article. The notion of classical norm relation has been studied by Biasse, Fieker, Hofmann and Page in \cite{norm}.
They use it to obtain inductive methods to compute the class group or the group of $S$-units of a Galois extension of number fields. The goal of this section is to generalise this notion in order to obtain a similar method applicable to more examples.
We will define our generalisation of norm relations in definition~\ref{def:GNR}, provide equivalent definitions in proposition \ref{pro:equiv_def}, and prove theorem~\ref{tm:criteria_compositum} which will be useful for the algorithms in section~\ref{sec:Algo}.

\begin{df}{}
Let $G$ be a finite group and $H$ a subgroup of $G$.
We call \emph{norm element} of $H$ the element $N_H = \sum_{h \in H} h \in \Z[G]$.
\end{df}

\begin{df}{}
Let $G$ be a finite group, $\J$ a set a subgroups of $G$ and $R$ a commutative ring. A \emph{norm relation over $R$ with respect to $\J$} is an equality in $R[G]$ of the form $$1 = \sum_{i = 1}^l a_i N_{J_i} b_i$$
where $a_i, b_i \in R[G]$, $J_i \in \J$, and $J_i \neq 1$.
\end{df}

\begin{ex}
    The symmetric group $G = S_3$ admits a norm relation over $\Q$ with respect to $\HH = \{ \langle (1, 2, 3) \rangle, \langle (2, 3) \rangle \}$.
\end{ex}

\begin{df}{} \label{def:GNR}
Let $G$ be a finite group, $H$ a subgroup of $G$, $\J$ a set a subgroups of $G$ and $R$ a commutative ring. A \emph{generalised norm relation over $R$ with respect to $H$ and  $\J$} is an equality in $R[G]$ of the form $$N_H = \sum_{i = 1}^l a_i N_{J_i} b_i$$
where $a_i, b_i \in R[G]$, $J_i \in \J$, and $J_i \neq 1$.
\end{df}

\begin{req}
Clearly, with the notations above, a classical norm relation is a generalised norm relation where $H$ is the trivial subgroup.
\end{req}

\begin{ex}
    The alternating group $A_4$ admits a norm relation over~$\Q$ with respect to $\HH = \{ C_2 \times C_2, C_3 \}$.
    It also admits a generalised norm relation with respect to $H = C_2$ and $\J = \{ C_2 \times C_2, C_3 \}$. Here, we see that the generalised norm relation comes from the regular norm relation. 
\end{ex} 

\begin{pro} \label{pro:equiv_def}
Let $G$ be a finite group, $H$ a subgroup of $G$, $\J = \{J_1, \cdots, J_\ell \}$ a set of non trivial subgroups of $G$ and $R$ a commutative ring. Then the following assertions are equivalent:
\begin{enumerate}
    \item\label{item:surjmorph} There exists a surjective morphism of $\Q[G]$-modules $$\phi\colon \bigoplus_{i = 1}^\ell \Q[G/J_i]^{n_i} \rightarrow \Q[G/H]$$
    where for all $i$, $n_i \in \N$.
    \item\label{item:idempotents} If $e_1, \dots, e_r$  are the central primitive idempotent elements of $\Q[G]$, then for all $1 \leq i \leq r$, if $e_i N_H \neq 0$, there exists $J \in \J$ such that $e_i N_J \neq 0$.
    \item\label{item:Qfixedpts} For all simple $\Q[G]$-module $V$, if $V^H \neq 0$, there exists $J \in \J$ such that $V^J \neq 0$.
    \item\label{item:Qbarfixedpts} For all simple $\overline{\Q}[G]$-module $V$, if $V^H \neq 0$, there exists $J \in \J$ such that $V^J \neq 0$.
    \item\label{item:Cfixedpts} For all simple $\C[G]$-module $V$, if $V^H \neq 0$, there exists $J \in \J$ such that $V^J \neq 0$.
    \item\label{item:twosidedideal} The norm element $N_H$ is in the two sided ideal $\langle N_J : J \in \J \rangle_{\Q[G]}$.
    \item\label{item:existnormrel} The group $G$ admits a generalised norm relation over $\Q$ with respect to $H$ and $\J$.
\end{enumerate}
\end{pro}

\begin{proof}\hfill

\begin{itemize}
    \item \ref{item:surjmorph} $\Rightarrow$ \ref{item:Qfixedpts}.
    We know there is an isomorphism of $R$-modules between $V^H$ and $\Hom_{\Q[G]}(\Q[G/H], V)$. 
    Likewise, for all $i$, $V^{J_i}$ is isomorphic to $\Hom_{\Q[G]}(\Q[G/J_i], V)$. Suppose \ref{item:surjmorph}, then we have the following diagram, where $f_H$ is an element of $V^H$ seen as an element of  $\Hom_{\Q[G]}(\Q[G/H], V)$, and likewise, the $f_{J_i}$ are elements of $\Hom_{\Q[G]}(\Q[G/J_i], V)$.

    \begin{center}
    \begin{tikzpicture}[->,>=stealth',shorten >=2pt,auto, semithick]
        \node (A) {$\Q[G/H]$};
        \node (B) [below right of=A, xshift = 7cm, yshift = -5.2cm] {$V$};
        \node (C) [below of=A, yshift = -5cm] {$\bigoplus_{i = 1}^r \Q[G/J_i]$};
        \path (A) edge [] node [above right] {$f_H \neq 0 \in V^H$} (B);
        \path (C) edge [] node [left] {$\phi$ surjective} (A);
        \path (C) edge [] node [below] {$f_H \circ \phi =: \sum_{i = 1}^r f_{J_i} \neq 0$} (B);
    \end{tikzpicture}
    \end{center}


    So $\sum_{i = 1}^r f_{J_i}$ is non zero, so at least one of the $f_{J_i}$ is non zero, hence the conclusion.

    \item \ref{item:idempotents} $\Leftrightarrow$ \ref{item:Qfixedpts}.
    Let $V_i$ be the simple $\Q[G]$-module (unique up to isomorphism) such that $e_iV_i \neq 0$ then $\Q[G]/(1-e_i)$ acts faithfully on $V_i$. So  $e_iN_H = 0$ if and only if  $N_H \cdot V_i = 0$, so if and only if $(\frac{1}{|H|}N_H) \cdot V_i = 0$ which is equivalent to  $V_i^H = 0$.

    \item \ref{item:Qfixedpts} $\Rightarrow$ \ref{item:surjmorph}.
    Suppose \ref{item:Qfixedpts}, then let  $V = \Q[G/H]$. Then $V$ is a $\Q[G]$-module, and $V$ can decompose as $V = \bigoplus_k V_k$, where the $V_k$ are simple. For all $k$, let $f_{k} \colon V \rightarrow V_k$ the projection. It can be seen as an element of $V_k^H$ by \ref{lm:fixed_points}. Then there exists a non zero element of $V_k^{J_i}$ for some $i$, by lemma \ref{item:Qfixedpts}. So we have a nonzero morphism $\bigoplus_{i = 1}^r \Q[G/J_i] \rightarrow V_k$ so it is surjective because $V_k$ is simple. Hence the conclusion by putting together all the $k$.
    
    \item \ref{item:Qfixedpts} $\Rightarrow$ \ref{item:Qbarfixedpts}.
    Suppose \ref{item:Qfixedpts}, let $W$ be a simple $\overline{\Q}[G]$-module. $W$ is isomorphic to a submodule of $V \otimes_\Q \overline{\Q}$, with $V$ a simple $\Q[G]$-module. Then we have $V \otimes_\Q \overline{\Q} \simeq \bigoplus_{j = 1}^k W_j$, where the $W_j$ are simple $\overline{\Q}[G]$-modules. So $W$ is isomorphic to one of the $W_j$.
    What's more, the $W_j$ are pairwise Galois conjugate, so $\dim_{\overline{\Q}}W_j^H = \dim_{\overline{\Q}}W_1^H$ pour tout $j$. So if $W^H$ is non zero, $V^H$ is also non zero. So, by~\ref{item:Qfixedpts}, there exists $J \in \J$ such that $V^J$ is non zero. Hence $W^J \neq 0$.

    \item \ref{item:Qbarfixedpts} $\Rightarrow$ \ref{item:Cfixedpts}.
    The simple $\C[G]$-modules are exactly the $V \otimes_{\overline{\Q}} \C$, where~$V$ is a simple $\overline{\Q}[G]$-module. The conclusion follows.

    \item \ref{item:Cfixedpts} $\Rightarrow$ \ref{item:Qbarfixedpts}.
    Suppose \ref{item:Cfixedpts}, let $V$ a simple $\overline{\Q}[G]$-module. If  $V^H \neq 0$, then $(V \times_{\overline{\Q}} \C)^H \neq 0$. So, by $4.$, there exists $J \in \J$ such that $(V \otimes_{\overline{\Q}} \C)^J \neq 0$. Hence $V^J \neq 0$.

    \item \ref{item:Qbarfixedpts} $\Rightarrow$ \ref{item:Qfixedpts}.
    Suppose \ref{item:Qbarfixedpts}, let $V$ a simple $\Q[G]$-module such that $V^H  \neq 0$. Consider $V \otimes_\Q \overline{\Q} \simeq \bigoplus_{j = 1}^k W_j$. We know that $W_j^H  \neq 0$ for all $j$. So there exists  $J \in \J$ such that $W_1^J  \neq 0$. So $V^J  \neq 0$.

    \item \ref{item:Qfixedpts} $\Leftrightarrow$ \ref{item:twosidedideal}.
    Let~$I$ be a two-sided ideal of $\Q[G]$. We have $I = \sum_{i = 1}^r e_i I$. What's more, there is an isomorphic projection of $e_i I$ in a two sided ideal of the algebra $\Q[G]/(1-e_i)$, which is simple. So $e_i I$ is either zero, or $e_i \Q[G]$. By applying this result to $I = \langle N_J : J \in \J \rangle_{\Q[G]}$, we find the equivalence.

    \item \ref{item:twosidedideal} $\Leftrightarrow$ \ref{item:existnormrel}.
    This equivalence comes directly from the definition of a generalised norm relation.
\end{itemize}

\end{proof}

\begin{req}
    It is relatively easy to apply these equivalent definition of generalised norm relations to implement algorithms that take a finite group $G$ and look for $H$ and $\J = \{ J_1, \cdots, J_l\}$ such that there is a generalised norm relation.
\end{req}

\begin{df}
Let $K,L_1, \cdots, L_\ell$ be number fields. Let $\Omega$ a Galois extension of $\Q$ containing~$K$ and all the $L_i$, and let $\G$ its Galois group. We denote by $\HH$ the subgroup of $\G$ fixing $K$, and by~$\Y_i$ the ones fixing the $L_i$ respectively. Then we say there is a generalised norm relation between~$K$ and the $L_i$ if there is a generalised norm relation over $\Q$ with respect to $\HH$ and the~$\Y_i$.
\end{df}

Now we want to use such a generalised norm relation to find an algorithm that can compute the class group of $K$ given the class class groups of all the $L_i$.
So we will only be interested in generalised norm relations where the degrees of the $L_i$ are lower than the degree of $K$. In other words, the order of all the $\Y_i$ has to be higher than the order of $\HH$.

\begin{ex}
    There is a generalised norm relation between the number field $K$ defined by $f(x) = x^6 - 6x^4 + 9x^2 + 23$ and the number fields $L_1, L_2$ respectively defined by $g_1(x) = x^3 - 9x - 27$ and $g_2(x) = x^2 + 207$. Indeed, $K/\Q$ is a Galois extension of Galois group $G = S_3$, and $L_1, L_2$ are the subgroups fixed respectively by $\langle (2, 3) \rangle$ and $\langle (1, 2, 3) \rangle$
\end{ex}

    
\begin{tm} \label{tm:galoisclosure}
Suppose there is a generalised norm relation between a number field $K$ and some $L_i$ that are not necessarily contained in the Galois closure $\Tilde{K}$ of $K$. Denote by $\Omega$ a Galois extension of $\Q$ of Galois group $\G$ containing $\Tilde{K}$ and all the $L_i$ as in the diagram below.

    
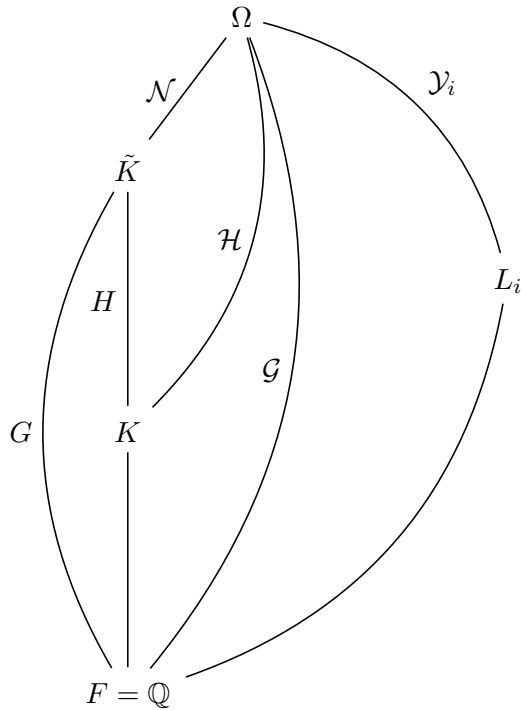
\begin{center}
    \begin{tikzpicture}[->,>=stealth',shorten >=2pt,auto, semithick]
        \node (A) {$\Omega$};
        \node (B1) [below of = A, yshift = -1 cm, xshift = -1.5cm] {$\tilde{K}$};
        \node (B2) [below of = B1, yshift = -2.5 cm] {$K$};
        \node (B3) [below of = B2, yshift = -2.5 cm] {$F = \Q$};
        \node (C) [right of = B1,yshift = -1.5 cm, xshift = 4 cm] {$L_i$};
        \path (A) edge [-] node [left] {$\NN $} (B1);
        \path (B1) edge [-] node [left] {$H$} (B2);
        \path (B2) edge [-] node [left] {} (B3);
        \path (B1) edge [-][bend right] node [left] {$G$} (B3);
        \path (A) edge [-][bend left] node [above right] {$\Y_i$} (C);
        \path (A) edge [-][bend left] node [left] {$\HH$} (B2);
        \path (A) edge [-][bend left] node [left] {$\G$} (B3);
        \path (C) edge [-][bend left] node [right] {} (B3);
    \end{tikzpicture}
    \end{center} 

Then there is also a generalized norm relation between $K$ and some $M_i$ that are contained in $\tilde{K}$.
\end{tm}

\begin{proof}

We have $\NN = \bigcap_{g\in \G} g \HH g^{-1}$. What's more, $\NN$ is normal in $\G$ and $G = \G/ \NN$ and $H = \HH/\NN$.
Since there is a generalised norm relation between the $L_i$ and $K$, there exists a relation of the form $N_{\HH} = \sum_i a_i N_{\Y_i} b_i \in \Q[G]$.
Consider the projection $$\pi\colon \Q[\G] \rightarrow \Q[\G/\NN] = \Q[G], \sum_i \lambda_i g_i \mapsto \sum_i \lambda_i \overline{g}_i.$$
This map $\pi$ is a surjective morphism of $\Q$-algebras. Composing the relation by  $\pi$ we get

$$\pi(N_{\HH}) = |\NN| N_H = \sum_i \pi(a_i) \pi(N_{\Y_i}) \pi(b_i)$$

and

$$\pi(N_{\Y_i}) = |\NN \cap \Y_i| N_{\Y_i/(\NN \cap \Y_i )}.$$

So there is a generalised norm relation between $K$ and the $M_i = \Omega^{\Y_i/ \NN} \in \tilde{K}$.
Note that if for some $i$, $\Y_i \subset \NN$, then~$\tilde{K} \subset L_i$, then the relation was not interesting.

\end{proof}

The following definition and properties, up to theorem \ref{tm:bornec} aim to provide a bound on a quantity that we call optimal coefficient, which will be useful for the proof of the complexity of algorithm \ref{algo:polynomial} in section~\ref{sec:Algo}.

\begin{df}{} \label{df: optcoeff}
Let $H, J_1, \cdots, J_\ell$ be  non trivial subgroups of $G$, and let $\J = \{J_1, \cdots, J_\ell\}$. If there is a norm relation over $\Q$ with respect to $H$ and $\J$, we define the \emph{optimal coefficient} $c(\J, H)$ to be the smallest positive integer such that there exists an injective morphism of $\Z[G]$-module $\psi\colon \Z[G/H] \rightarrow \bigoplus_{i} \Z[G/J_i]^{n_i}$ with $n_i \in \N$ for all $i$, and a morphism of $\Z[G]$-module $\phi\colon \bigoplus_{i} \Z[G/J_i]^{n_i} \rightarrow \Z[G/H]$ such that the image of $\phi$ has finite index in $\Z[G/H]$ and  $\phi \circ \psi = c(\J, H) \cdot \id$.
\end{df}

To prove that the optimal coefficient is well defined, we start by giving a more general proposition.

\begin{pro} \label{pro:pseudo_inv}
    Let $\Gamma$ be a finite group, and let $\HH_1, \cdots, \HH_r, \Y_1, \cdots, \Y_s$ be some subgroups of $\Gamma$. Let $M = \bigoplus_i \Q[\Gamma/\HH_i]$ and $N = \bigoplus_j \Q[\Gamma/\Y_j]$.
    \begin{enumerate}
        \item \label{surj} If there exists a surjective morphism of $\Q[\Gamma]$-modules $$\Phi \colon M \rightarrow N,$$ then there is an injective morphism of $\Q[\Gamma]$-modules $$\Psi \colon  N \rightarrow M$$ such that $\Psi \circ \Phi = \id_{N \rightarrow N}$.
        \item \label{inj} Similarly, If there exists an injective morphism of $\Q[\Gamma]$-modules $$\Psi \colon N \rightarrow M,$$ then there is a surjective morphism of $\Q[\Gamma]$-modules $$\Phi \colon M \rightarrow N$$ such that $\Psi \circ \Phi = \id_N$.
    \end{enumerate}
\end{pro}

\begin{proof}

Let us prove \ref{surj}. 
Since $\Q[\Gamma]$ is a semi-simple algebra, this means we can write the decomposition in simple modules. Up to isomorphism, $N = \bigoplus_{j = 1}^n W_j$ and $M = \bigoplus_{j = 1}^n W_j \oplus \bigoplus_{k = 1}^m V_k$, and $\Phi$ is the projection. Then $\Psi$ is the natural injection $\bigoplus_{j = 1}^n W_j \rightarrow \bigoplus_{j = 1}^n W_j \oplus \bigoplus_{k = 1}^m V_k$.

The proof of \ref{inj} is similar.

\end{proof}

\begin{pro} \label{pro:optimal_coef}
With the notations of the definition above, $c(\J, H)$ is well defined.
\end{pro}

\begin{proof}
Since there is  a norm relation over $\Q$ with respect to $H$ and~$\J$, there is a surjective $\Q[G]$-module morphism $ \bigoplus_{i} \Q[G/J_i]^{n_i} \rightarrow \Q[G/H]$.

Consider $\Psi$ the injection given by proposition \ref{pro:pseudo_inv}, let $c$ be the LCM of the denominators of all coefficients of all the $\Psi (gH)$ for $gH \in G/H$.
Then $c \cdot \Psi$ induces an injective morphism of $\Z[G]$-module $\Z[G/H] \rightarrow \bigoplus_i \Z[G/J_i]^{n_i}$.
With the same reasoning, we can construct a morphism of $\Z[G]$-modules
$\Phi\colon \bigoplus_i \Z[G/J_i]^{n_i} \rightarrow \Z[G/H]$ whose image is of finite index in $\Z[G/H]$. And then $\Psi \circ \Phi$ is a multiple of $\id_{\Z[G/H]}$. Hence the conclusion.

\end{proof}

We now prove that the optimal coefficient is also smallest for the divisibility relation.

\begin{pro}
If $c$ is a positive integer such that there exists $\phi$ and $\psi$ as in definition \ref{df: optcoeff} such that $\phi \circ \psi = c \cdot \id_{\Z[G/H]}$, then $c(\J, H) \mid c$.
\end{pro}

\begin{proof}

Consider the group $$ E =  \langle t_2 \circ t_1 |  n_i \in \N^\times \forall i, t_1 \in A_{1, n_i}, t_2 \in A_{2, n_i} \rangle_\Z \cap \Z \id_{\Z[G/H]}, $$

where $$A_{1, n_i} = \Hom_{\Z[G]}(\Z[G/H], \bigoplus_i \Z[G/J_i]^{n_i})$$
and $$A_{2, n_i} = \Hom_{\Z[G]}( \bigoplus_i \Z[G/J_i]^{n_i}, \Z[G/H]) .$$

Then $E$ is a subgroup of $\End_{\Z[G]}(\Z[G/H] )$ contained in $\Z \id$, so $E$ is of the form $a \Z \id$ with $a \in \N$. And by definition, $a = c(\J, H)$.
By construction, $c \cdot \id$ is in $E$, hence $c(\J, H) \mid c$.

\end{proof}


\begin{tm} \label{tm:bornec}
With the notations of Definition~\ref{df: optcoeff}, we have $c(\J, H) \mid |G|^2$.
\end{tm}

\begin{proof}

Let $p$ be a prime number. Let $\O$ be a maximal order of $\Q_p[G]$ containing $\Z_p[G]$. By \cite[27.1, proposition]{rep_th} we have $ \O \subset \frac{1}{|G|} \Z_p[G]$.

Consider $M_H = \O \cdot \Z_p[G/H] \subset \Q_p[G/H]$. Then $M_H$ is an $\O$-module, and we have $\Z_p[G/H] \subset M_H \subset \frac{1}{|G|} \Z_p[G]$.
Similarly, for all $i$, we write $M_{J_i} = \O \cdot \Z_p[G/J_i]$.

Let $e_1, \cdots, e_r$ be  central primitive idempotents of $\Q_p[G]$ contained in $\O$, which exist since $\O$ is a maximal order.
For all $1 \leq i \leq r$, there is a isomorphism $\alpha\colon \O/(1-e_i) \rightarrow M_n(\Lambda)$, where $\Lambda = \Lambda_i$ is the maximal order of
a division algebra $D$ over $\Q_p$.
And $\alpha$ can be extended to $\O$ with the projection $\O \rightarrow \O/(1 - e_i)$. (see \cite{max_orders}).

We have $M_n(\Lambda) \subset M_n(D)$ and $M_n(D)$ acts on $D^n$, which is the only simple $M_n(D)$-module up to isomorphism.

So $M_H \otimes \Q_p \cong D^{na}$ with $a \in \N^\times$, since $M_H \otimes \Q_p$ is a is a $M_n(D)$-module. Similarly, $\bigoplus_i M_{J_i}^{n_i} \otimes \Q_p \cong D^{nb}$, and thus $M_H \cong \Lambda^{na}$ and $\bigoplus_i M_{J_i}^{n_i} \cong \Lambda^{nb}$.

What's more, we have an surjective  morphism of $\Q_p[G]$-modules from $\bigoplus_i M_{J_i}^{n_i} \otimes \Q_p = \bigoplus_i \Q_p[G/J_i]$ to $M_H \otimes \Q_p = \Q_p[G/H]$. Which means that $a \leq b$.

Therefore, we have a natural injection of $\O$-modules $\Tilde{\psi} \colon M_H \rightarrow \bigoplus_i M_{J_i}^{n_i}$, and a natural surjection $\Tilde{\phi} \colon \bigoplus_i M_{J_i}^{n_i} \rightarrow M_H$.

Let us denote $\psi = |G| \Tilde{\psi}$ and $\phi = |G| \Tilde{\phi}$.
That way, $\psi$ induces an injective morphism $\Z_p[G/H] \rightarrow \bigoplus_i \Z_p[G/J_i]^{n_i}$ and $\phi$ a surjective morphism $\bigoplus_i \Z_p[G/J_i]^{n_i} \rightarrow \Z_p[G/H]$. And we have $\phi \circ \psi = |G|^2 \id$.

By doing the same reasoning over all $e_i$ and by putting together every prime $p$, we obtain the claimed result.

\end{proof}

\begin{req}
    In the algorithms of section \ref{sec:Algo}, the optimal coefficient $c(\J, H)$ plays a role analogous to that of the optimal denominator $d(\HH)$ in the case of classical norm relations (see \cite[definition 2.15]{norm}). Thanks to our new definition, we obtain a $|G|^2$ bound instead of the $|G|^3$ bound in \cite[theorem 2.20]{norm}.
\end{req}

Using the isomorphisms in section \ref{sec:HeckeCompo}, we want to find an equivalent definition of generalised norm relations that features only field theory and not group theory. This definition will be useful to design efficient algorithms that will not require the computation of Galois groups or Galois closures.

\begin{lm}\label{lm:surj}
Let $V$ be a $R[G]$-module and $\phi\colon V \rightarrow R[G/H]$ a surjective morphism of $R[G]$-modules. There exists a preimage of $1H$ by $\phi$ that is in $V^H$.
\end{lm}

\begin{proof}

Since $\phi$ is surjective, there exists $v \in V$ such that $\phi(v) = 1H$. Now consider the element $v' = \frac{1}{|H|} \sum_{h \in H} h \cdot v$.

Then, clearly, $v' \in V^H$, and $\phi(v') = \frac{1}{|H|} \sum_{h \in H} \phi(h \cdot v) =  \frac{1}{|H|} \sum_{h \in H} h \cdot \phi(v) = \frac{1}{|H|} \sum_{h \in H} h \cdot 1H = 1H$.

\end{proof}

\begin{pro} \label{pro:CritHecke}

We have a generalised norm relation, given by a surjection $$\phi\colon \bigoplus_i R[G/J_i] \rightarrow R[G/H],$$  if and only if there is a relation of the form $$1H = \sum_i T_{\sum_h \mu_{i,h} J_i \delta_{i,h} H} \sum_k \lambda_{i,k} g_{i,k} J_i,$$ with $\sum_k \lambda_{i,k} g_{i,k} J_i $ in $(\bigoplus_i R[G/J_i])^H$.
\end{pro}

\begin{proof}

Suppose there exists $\phi\colon \bigoplus_i R[G/J_i] \rightarrow R[G/H]$ surjective. Let us considere $\bigoplus_i \sum_k \lambda_{i,k} g_{i,k} J_i$ a preimage of $1H$. By lemma \ref{lm:surj}, we can suppose $\bigoplus_i \sum_k \lambda_{i,k} g_{i,k} J_i$ is in $(\bigoplus_i R[G/J_i])^H$ 

Let us write $\phi = \bigoplus_i \phi_i$ with $\phi_i\colon \Q[G/J_i] \rightarrow \Q[G/H]$. Then we have $$1H = \sum_i \phi_i (\sum_k \lambda_{i,k} g_{i,k} J_i ).$$
Then, by writing  $\phi_i = \sum_h \mu_h T_{J_i \delta_{i, h} H} = T_{\sum_h \mu_{i,h} J_i \delta_{i,h} H}$, we can obtain 
$$1H = \sum_i T_{\sum_h \mu_{i,h} J_i \delta_{i,h} H} \sum_k \lambda_{i,k} g_{i,k} J_i $$

\end{proof}

\begin{cor}
Let $S$ be set of non-zero prime ideals of $\O_K$.
The map $$\Phi\colon \bigoplus_{i = 1}^\ell \bigoplus_{C \in \Compos(K_i, K)} \O_{K_j, S}^\times \rightarrow \O_{K, S}^\times$$ $$ \bigoplus_{i = 1}^\ell \bigoplus_{C \in \Compos(K_i, K)} \mathfrak{a}_{i, C} \mapsto \sum_{i = 1}^\ell \sum_{C \in \Compos(K_i, K)} C \cdot \mathfrak{a}_{i, C} $$ has an image of maximal rank.

\end{cor}

\begin{tm} \label{tm:criteria_compositum}
If $L_1, \cdots , L_\ell$ are number fields, and $\beta_1, \cdots, \beta_\ell$ such that $L_i = \Q(\beta_i)$, then $K$ admits a \emph{generalised norm relation with respect to $L_1, \cdots , L_\ell$}, if and only if there is a relation of the form
$$\alpha = \sum_{i = 1}^{\ell} \sum_{C \in Compo(K, L_i)} a_{i, C} C \cdot \beta_i$$
where the coefficients $a_{i,C}$ are in $\Q$.
\end{tm}

\begin{proof}
This theorem is a rephrasing of proposition $\ref{pro:CritHecke}$ using the isomorphisms of part \ref{sec:HeckeCompo}.
\end{proof}

\section{Mackey Functors} \label{sec:Mackey}

In this section, we will recall some properties of cohomological Mackey functors. They will be useful mainly to prove the correctness of some algorithms in section \ref{sec:Algo}. The results in this section come from \cite{Boltje} and \cite{Yoshida}.

The main result in this section will be proposition \ref{pro:Mackeyfct}, which we will later use to find a relation between the $S$-units of the number fields involved in a generalised norm relation.

First let us recall the definition of a Mackey functor, as in \cite{Boltje}.

\begin{df}
    Let $G$ be a finite group and $R$ a commutative ring. A $R$-\emph{Mackey functor} $M = (M, c, \Res, \Ind)$ on $G$ is a quadruple consisting of
    \begin{itemize}
        \item a family of $R$-modules $M(H)$ where $H \leq G$,
        \item a family of homomorphisms of $R$-modules $c_{g, H}\colon M(H) \rightarrow M(^gH)$, the \emph{conjugation} maps, where $g \in G$, $H \leq G$ and $^gH = g H g^{-1}$,
        \item a family of homomorphisms of $R$-modules $\Res_J^H \colon M(H) \rightarrow M(J)$, the \emph{restriction maps}, where $J \leq H \leq G$, and
        \item a family of homomorphisms of $R$-modules $\Ind_J^H \colon M(J) \rightarrow M(H)$, the \emph{induction} maps, where $J \leq H \leq G$,
    \end{itemize}
    such that the following axioms are satified:
    \begin{itemize}
        \item (Triviality) $c_{h, H}= \Res_H^H = \Ind_H^H = \id_{M(H)}$ for all $H \leq G$ and $h \in H$.
        \item (Transitivity) $c_{g'g, H} = c_{g', ^gH} \circ c_{g, H}$, $\Res_L^J \circ \Res_J^H = \Res_L^H$ and $\Ind_J^H \circ \Ind_L^J = \Ind_L^H$ for all $L \leq J \leq H \leq G$ and $g, g' \in G$.
        \item ($G$-equivariance) $c_{g, J} \circ \Res_J^H = \Res_{^gJ}^{^gH} \circ c_{g, H}$ and  $c_{g, J} \circ \Ind_J^H = \Ind_{^gJ}^{^gH} \circ c_{g, J}$ for all $J \leq H \leq G$ and $g \in G$.
        \item (Mackey formula) For all $H \leq G$, $U, J \leq H$, one has $$ \Res_U^H \circ \Ind_J^H = \sum_{h \in U \backslash H / J} \Ind_{U \cap ^hJ}^U \circ \Res_{U \cap ^hJ}^{^hJ} \circ c_{h, J} $$ where $h \in H$ runs through a set of representatives for the double cosets $U \backslash H / K$.
    \end{itemize}
\end{df}

\begin{df}
    A $R$-Mackey functor $M$ on $G$ is called \emph{cohomological} if the axiom $$\Ind_J^H \circ \Res_J^H = [H:J] \id_{M(H)}, \mbox{ for all } J \leq H \leq G$$ holds. 
\end{df}

\begin{tm}
    Let $R$ be a commutative ring and $G$ a group. The association $H\mapsto R[G/H]$ forms a cohomological Mackey functor with the following operations:
    \begin{itemize}
        \item $\Ind_K^H\colon R[G/K] \rightarrow R[G/H], gH \mapsto gK$
        \item $\Res^H_K\colon R[G/H] \rightarrow R[G/K], gH \mapsto \sum_{h \in H/K} ghK$
        \item $c_{g,H}\colon R[G/H] \rightarrow R[G/^gH], xH \mapsto x g^{-1} ~^gH$
    \end{itemize}
\end{tm}

\begin{proof}
See \cite[example 4.1]{Yoshida}, with $D$ the trivial group.
\end{proof}

\begin{tm}\label{tm:THgK}
Let $M$ be a cohomological Mackey functor.
If $H,K$ are subgroups of $G$ and $g$ an element of $G$, let us define the operator $$T_{HgK}\colon M(K) \rightarrow M(H), x \mapsto \Ind_{^gK\cap H}^H \circ \Res_{^gK \cap H}^{^gK} \circ c_{g,K} (x) .$$
Then, all operators of this form follow the rules of compositions of $R[H\backslash G / K]$ comming from the isomorphism of proposition \ref{prop:Heckemorphisms}.
\end{tm}

\begin{proof}
See \cite[theorem 4.1]{Yoshida}.
\end{proof}

\begin{pro}\label{pro:Mackeyfct}
    Let $R$ be a ring, $G$ a group, $H < G$ a subgroup and $\{J_i\}$ a set of subgroups.
    If we have $\phi\colon \bigoplus_{i=1}^m R[G/J_i] \rightarrow R[G/H]$ a morphism of $R[G]$-modules and $\psi\colon R[G/H] \rightarrow \bigoplus_{i=1}^m R[G/J_i]$ a morphism of $R[G]$-modules, such that $\phi \circ \psi = d \cdot \id_{R[G/H]}$, then for every cohomological Mackey functor $M$, there exists $\phi_M\colon \bigoplus_{i=1}^m M(J_i) \rightarrow M(H)$ and  $\psi_M\colon M(H) \rightarrow \bigoplus_{i=1}^m M(J_i)$ such that $\phi_M \circ \psi_M = d \cdot \id_{M(H)}$.
\end{pro}

\begin{proof}
See \cite[corollary 1.4]{Boltje}.
\end{proof}

\begin{req}\label{req:decomp}

We can describe more precisely the form of $\phi_M$ and $\psi_M$. They are obtained by decomposing $\phi$ and $\psi$ into sums of morphisms respectively $R[G/J_i] \rightarrow R[G/H]$ and $R[G/H] \rightarrow R[G/J_i]$, expressing these morphisms as elements of $H \backslash G / J_i$ or $J_i \backslash G / H$ and then applying theorem \ref{tm:THgK}.

\end{req}

\begin{req} \label{req:S_units}
This previous proposition, along with proposition \ref{pro:optimal_coef}, give a induction relation between $M(H)$ and the $M(J_i)$ for every cohomological Mackey functor $M$. We will use it in section \ref{sec:Algo} with $M(H) = \O_{\Tilde{K}^H, S}^\times$ and $M(J_i) = \O_{\Tilde{K}^{J_i}, S}^\times$, but it could also be useful to compute other Mackey functors.
\end{req}

\section{Algorithms} \label{sec:Algo}

In this section, we will present algorithms to resolve some problems around the notion of generalised norm relation.

We will suppose the Galois group and the Galois closure of the fields we will use are not known a priori.

Note that if we know the Galois group of a field, it is easy to make an algorithm that determines all the generalised norm relation. We implemented such an algorithm and it has been useful to find examples of generalised norm relations. There is also a method to calculate the group of $S$ units of $\tilde{K}^H$ using the $S$-units of the $\Tilde{K}^{J_i}$ if $\tilde{K}$ is a Galois extension of $\Q$ of Galois group $G$ and $G$ admits a generalised norm relation over $\Q$ with respect to~$H$ and~$\J = \{J_1, \cdots, J_\ell \}$.

Here, we will describe algorithms relying only on field theory, and without having to compute any Galois group which we do not know how to compute in polynomial time. 

The main algorithm in this section is algorithm \ref{algo:polynomial}, which computes a $\Z$-basis of the $S$-units of a number field $K$, given some fields $K_j$ such that $K$ admits a generalised norm relation with respect to the $K_j$. Its complexity is polynomial in the size of the input, including a $\Z$-basis of the $S$-units of the $K_j$ (see theorem \ref{tm:algopoly}).

\begin{algo}

\underline{input:} A number field $K = \tilde{K}^H$ and a family $(K_i = \tilde{K}^{J_i})$ of number fields given by~$f$, the minimal polynomial of $\alpha$ with  $K = \Q(\alpha)$, and $f_i$ the minimal polynomials of the~$\beta_i$, with $K_i = \Q(\beta_i)$.

\underline{output:} A boolean indicating whether there is a generalised norm relation, and if so, a formula of the form 
$$1H = \sum_i T_{\sum_h \mu_{i,h} J_i \delta_{i,h} H} \sum_k \lambda_{i,k} g_{i,k} J_i$$ in $\Z[G/H]$

\begin{itemize}
    \item For all $i$, list all compositums of $K$ and $K_i$.
    
    If $f_i = p_1 \cdots p_r \in K[X]$, Then, the compositums are the $K[X]/ (p_j)$, with $\iota_K$  the inclusion, and $\iota_{L_i}\colon \beta_i \mapsto X$.
    \item For all $i$, and for all $\sigma \in \Hom(L_i, \C)$ and for every compositum~$\CC$, compute $\CC \cdot \sigma \in \Q[\Hom(K, \C)]$.
    \item By linear algebra in $\Q[G/H] = \Q[\Hom(K, \C)]$, find a linear combination of these element that amounts to $1H$ (if such a combination exists).
\end{itemize}

\end{algo}

\begin{tm}
This algorithm is correct and its complexity is polynomial in the size of the input.
\end{tm}

\begin{proof}
The correctness of the algorithm follows from theorem \ref{tm:criteria_compositum}.

For the complexity, we have to check that every step of the algorithm works in polynomial time.
\begin{itemize}
    \item Listing all the compositums boils down to a problem of factorisation of polynomials in $K[X]$, which is polynomial thanks to the LLL algorithm (see \cite{LLL}).
    The number of compositums to list is at most $\sum_{j = 1}^\ell \deg(K_j)$.
    \item Given a complex embedding $\sigma$ of a field $K_j$, and a compositum $\CC$ of~$K$ and $K_j$, computing $\CC \cdot \sigma$  is in $\O(\deg(K_j) \times \deg(K))$. And the number of times such a computation occurs is at most $\sum_j \deg(K_j) \times |\Compos(K, K_j)|$.
    What's more, the size of $\CC \cdot \sigma$ is polynomial in the size of the input.
\end{itemize}

\end{proof}

\begin{algo} \label{algo:polynomial}
\underline{input:} A number field $K$ and a set of number fields $\{K_j\}$, each given by an irreducible polynomial in $\Q[X]$ and such that $K$ admits a generalised norm relation with respect to the $K_j$, a set $S$ of prime numbers, and for each $j$ a $\Z$-basis $B_j$ of $\O_{K_j, S}^\times$.

\underline{output:} A $\Z$-basis of $\O_{K, S}^\times$.

\begin{enumerate}

    \item \label{item:alg_primes} Compute $\pi_1, \cdots, \pi_k$ all the prime divisors of $(n!)^2$ where $n$ is the degree of $K$ (ie all the primes lesser than $n$). Let $r_i = v_{p} ( (n!)^2 )$.

    \item \label{item:alg_compos} For all $j$, compute all the compositums of $K$ and $K_j$ (up to isomorphism).

    \item \label{item:alg_im_compos} Compute the set $B$ of images of every element of the $B_j$ by every compositum of $K$ and $K_j$.

    \item \label{item:alg_subgroup} Compute the subgroup $V \subset \O_{K, S}^\times$ generated by $B$.

    \item \label{item:alg_boucle} For every $i$:
    \begin{itemize}
        \item $V_i \leftarrow V$
        \item $V_i \leftarrow \langle V_i, (\alpha_1)^{\frac{1}{p_i} }, \cdots, (\alpha_m)^{\frac{1}{p_i} } \rangle $ where $(\overline{ \alpha_i })$ is a basis of $(V_i \cap (K^\times)^{p_i} )/ V_i^{p_i}$. (See \cite[corollary 4.13]{norm})
        \item Reduce the basis of $V_i$ as in \cite[lemma 7.1]{lattice}.
    \end{itemize}

    \item \label{item:alg_concat} $V \leftarrow V_1 \cdots V_k$

    \item \label{item:alg_return} Return a basis of $V$.
    
\end{enumerate}

\end{algo}

\begin{tm} \label{tm:algopoly}
Assume the generalized Riemann Hypothesis (GRH). Then this algorithm is correct and its complexity is polynomial in the size of the input.
\end{tm}

\begin{proof}
First, let us prove the correctness. Let $G$ be the Galois group of $K$, let $H$ the subgroup fixing $K$ and for every $i$, let~$J_i$ the subgroup fixing $K_i$. Since there is a generalised norm relation, we know that there exists an integer $c$, a surjective morphism of $\Z[G]$-module $\phi\colon \bigoplus_i \Z[G/J_i] \rightarrow \Z[G/H]$ and an injective morphism of $\Z[G]$-modules $\psi\colon \Z[G/H] \rightarrow \bigoplus_i \Z[G/J_i]$, such that $\phi \circ \psi = \id$ (by proposition \ref{pro:optimal_coef}). 

Therefore, by proposition \ref{pro:Mackeyfct}, for any cohomological Mackey functor $M$, there is a surjective morphism $\phi_M\colon \bigoplus_{i=1}^m M(J_i) \rightarrow M(H)$. Consider $M(H) = \O_{\Tilde{K}^H, S}$ and $M(J_i) = \O_{\Tilde{K}^{J_i}, S}$. Since we know by remark \ref{req:decomp} that $\phi_M$ can be expressed as a sum of elements of $J_i \backslash G / H$, and since, by proposition \ref{pro:HeckeCompo}, these can be seen as elements of $\Compos(K_i, K)$, this proves the correctness.

Then let us prove the complexity. 
Let $\Sigma$ denote the total size of the input. To compute all the $\pi_i$ in step \ref{item:alg_primes}, we can use a sieve method, which is polynomial in $n$ where $n$ is the degree of $K$. Therefore, step \ref{item:alg_primes} takes polynomial time.

As seen before, for every $j$, computing all the compositums of $K$ and $K_j$ takes polynomial time. What's more, the number and the size of the compositums obtained are also polynomial. So step \ref{item:alg_compos} is also polynomial.

The size of the image of an element $x \in K_j$ by a compositum $\CC$ is also polynomial, since the map induced by~$\CC$ is the composition of the injection $K_j \rightarrow C$ and the norm $C \rightarrow K$. So step \ref{item:alg_im_compos} is polynomial.

For step \ref{item:alg_subgroup} as well as step \ref{item:alg_return},  one can deduce a basis from a generating set of the groups involved in polynomial time. The algorithms of \cite{Ge} provide a basis of the
 relations between the generators, and the Hermite normal form \cite{HNF} allows us to
 obtain a basis of the group in polynomial time.
 
The saturation in step \ref{item:alg_boucle} is performed as many times as the number of primes dividing $(n!)^2$, counted with multiplicity, according to theorem \ref{tm:bornec}. That number is polynomial in $n$, since the number of different primes in the decomposition of $(n!)^2$ is at most $n$, and for every prime $p$, $v_p(n!) \leq \frac{\log(n!)}{\log(2)}  = \O(n \log(n))$.

\end{proof}

\begin{req}
    The paper~\cite{prod_HR} gives a polynomial method to approximate $\kappa_K$, the residue of the Dedekind zeta function $\zeta_K(s)$ at $s = 1$ of a number field~$K$, from the discriminant $\Delta_K$ and the norm of prime ideals of $K$.
\end{req}

We now present an alternative to Algorithm~\ref{algo:polynomial}, which is more efficient in practice but not provably polynomial-time.

\begin{algo} \label{algo:faster}

\underline{input:} A number field $K = \tilde{K}^H$ and a family $(K_i = \tilde{K}^{J_i})$ of number fields, such that $K$ admits a generalised norm relation with respect to $K_1, \cdots K_\ell$. We know $f$, the minimal polynomial of $\alpha$ with  $K = \Q(\alpha)$, and $f_i$ the minimal polynomials of the  $\beta_i$, with $L_i = \Q(\beta_i)$.

\underline{output:} The structure of the class group of $K$

\begin{enumerate}
    \item For every $K_j$, compute every compositums of $K$ and $K_j$.


    \item Compute $HR_K = h_K \Reg_K$ using the approximation method in~\cite{prod_HR}. An approximation up to a factor 1.5 is enough.

    \item Initialize $T$ a set of prime ideals $\mathfrak{p}$ such that $N(\mathfrak{p}) = 1 \mod d$, where $d = \deg(K)^2$.
    
    \emph{The primes in $T$ will be used to detect $d$-th powers.}

    \item Initialize $S_\Q$ a set of prime numbers, and compute the set $S$ of prime ideals of $K$ above the primes in $S_\Q$.
    
    \emph{We hope that $S$ will generate the class group.}

    \item For all $K_j$, let $S_j$ be the set of prime ideals of $K_j$ above all primes $p$ in $S_\Q$, and compute a set $U_j$ of generators of the group of $S_j$ units of $K_j$. 

    \item For each $j$, for each $\mathfrak{p}$ in $S_j$, compute the vector $V_{j, \mathfrak{p}}$ of valuations of every element of $U_j$ at $\mathfrak{p}$. 

    \item Compute the matrix of a map $\Phi$, that sends all the ideals above all the primes in $S_j$ to their image by every compositum. Apply this matrix to every $V_{j, \mathfrak{p}}$, then concatenate all the vectors to obtain a matrix $M$.

    \item Apply the action of every compositum to every generator of the $U_j$ then compute the discrete logarithms in $\F_\mathfrak{p}$ of for every $p$ in $T$. Concatenate all the vectors of discrete logarithms to obtain a matrix $N$.
    
    \item Concatenate the matrix $M$ and $N$ and compute the kernel $R$ modulo $d$ of this matrix.
    
    \emph{We hope to obtain a basis of the $d$-saturation of the images in $K$ of the $S_j$-units of the $K_j$ by the actions of every compositum.}

    \item Compute the Smith normal form of the concatenation of $M$ and a basis of $R$.
    
    \emph{If $T$ and $S$ are large enough, that should give us the structure of $\Cl(K)$}.

    \item Compute the regulator of the group of units of $K$ obtained by the $d$-saturation of images of the units of the $K_j$ by the actions of every compositum. Multiply it with the class number to obtain a new $HR$ product, that we will denote by $\Tilde{HR_K}$. If the approximation for $HR_K$ is up to a factor $1.5$, then the regulator should be calculated with precision up to a factor $\frac{4}{3}$.

    \item Check if the $HR$ product corresponds to the one in step 3. If not, increase the size of $T$ and $S_\Q$ and go back to step 5.
    
\end{enumerate}

\end{algo}

\begin{tm}
    If this algorithm terminates, then it is correct.
\end{tm}

\begin{proof}
By the remark \ref{req:S_units}, we know the algorithm finds indeed all the $S$-units in $K$. Then, if the verification of the $HR$ product is correct, it means the $S$-units are enough to generate the class group.
The crucial observation is that the approximation errors due to the choice of $T$ and $S$ cannot compensate. If $T$ is not large enough and the algorithm incorrectly assumes an element to be a $d$-th power, then $\Tilde{HR_K}$ a divisor of its expected value. The same will happen if $S$ is not large enough to generate the class group.
\end{proof}

\begin{req}
Suppose we have a number field $K = \tilde{K}^H$ and a family $(K_i = \tilde{K}^{J_i})$ of number fields, such that $K$ admits a generalised norm relation with respect to $K_1, \dots, K_\ell$. If we want to compute the class group of $K$ using algorithm~\ref{algo:faster}, we could expect the most expensive step to be the computation of the $S_j$ units in all the $K_j$, since it is the only step whose computation is not polynomial in the size of the input.
However, in practice, when we try to apply this method to reasonable size examples, the most expensive step is often the computation of the images of the ideals in the $S_j$ by the compositums. 
\end{req}

\section{Comparison with classical norm relations} \label{sec:compar}

In this section, we will discuss the relevance of studying generalised norm relation instead of classical norm relation. A generalised norm relation of a group $G$ with respect to $H<G$ and a set of subgroup~$\J$ can come directly from a classical norm relation in $G$ (see fact \ref{req:subgroups}) or in a quotient of $G$ (see proposition \ref{pro:quotient}). But we will see that it is not always the case, and that in some examples, the methods in section \ref{sec:Algo} indeed allows to compute the class groups more efficiently than classical norm relations.

\begin{ft} \label{req:subgroups}
If there is a classical relation $1 = \sum_{i = 1}^l a_i N_{J_i} b_i$ for some finite group $G$ and some set~$\J$ of subgroups of $G$, then for any subgroup $H$, we can construct a generalised norm relation with respect to $H$ and $\J$, simply by multiplying both sides of the classical relation by $N_H$.
\end{ft}

\begin{pro} \label{pro:quotient}
Let $G$ be a finite group, $H, J_1, \cdots, J_l$ subgroups of~$G$. Let $N$ be a normal subgroup of $G$ contained in $H$. Denote by $\pi$ the projection from $G$ to $G/N$. Then $G$ admits a generalised norm relation with respect to $H$ and $J_1, \cdots, J_\ell$ if and only if $G/N$ admits a generalised norm relation with respect to $\pi (H)$ and $\pi(J_1), \cdots, \pi(J_\ell)$.
\end{pro}

\begin{proof}

Suppose $G$ admits a generalised norm relation over  $\Q$ with respect to $H$ and $J_1, \cdots, J_\ell$, of the form $N_H = \sum_{i = 1}^l a_i N_{J_i} b_i$.

Let $\Pi \colon \Q[G] \rightarrow \Q[G/N], \sum_i \lambda_i g_i \mapsto \sum_i \lambda_i \pi(g_i)$. Then $\Pi$ is a surjective morphism of $\Q[G]$-modules. And we have $\Pi (N_H) = |N| N_{H/N}$, and $\Pi (N_{Ji}) = |N \cap J_i | N_{J_i / (N \cap J_i) }$. Then, if we compose the relation by $\Pi$, we get a generalised norm relation of $G/N$  with respect to $\pi (H)$ and $\pi(J_1), \cdots, \pi(J_\ell)$. 

Now suppose $G/N$ admits a generalised norm relation with respect to $\pi (H)$ and $\pi(J_1), \cdots, \pi(J_\ell)$. So there is a surjective morphism $\phi\colon \bigoplus_{i=1}^l \Q[ \pi(G) / \pi(J_i)] \rightarrow \Q[ \pi(G) / \pi (H) ]$.

So $\phi \circ \Pi$ is a surjective morphism from $\bigoplus_{i=1}^l \Q[ G / J_i]$ to $\Q[\pi(G)/\pi(H)]$. And since $N \subset H \subset G$, we have $\pi(G) / \pi (H) \simeq G/H$. Hence a surjective morphism from $\bigoplus_{i=1}^l \Q[ G / J_i]$ to $\Q[G/H]$.

\end{proof}

It is important to note however that some generalised norm relation do not come from a regular norm relation in a subgroup or in a quotient.

\begin{ex} \label{ex: quotient/subgroup}

    
    For example, the symmetric group $S_4$ admits a norm relation over~$\Q$ with respect to $H = C_2 \times C_2$, and $\J = \{ D_8, S_3 \}$.
    This generalised norm relation does not come from a regular norm relation because we can check that $S_4$ does not have a norm relation with respect to $\J$. It does not come from a quotient either because the largest normal subgroup of $S_4$ contained in $H$ is trivial.
\end{ex}

Classical norm relations can be useful to compute class groups of number fields even when they are not Galois extensions of $\Q$. 

Indeed, let $K$ be a non Galois extension of $\Q$. Denote by $\Tilde{K}$ its Galois closure, $G$ its Galois group and $H < G$ such that $K = \Tilde{K}^H$. Suppose there is a subfield $L$ of $\Tilde{K}$ and a subfield $M$ of $L$ such that $L/M$ is a Galois extension of Galois group $\Gamma$. Suppose also that there exists a classical norm relation in $\Q[\Gamma]$ involving some subgroups $\Delta_i$, as in the figure below.

    
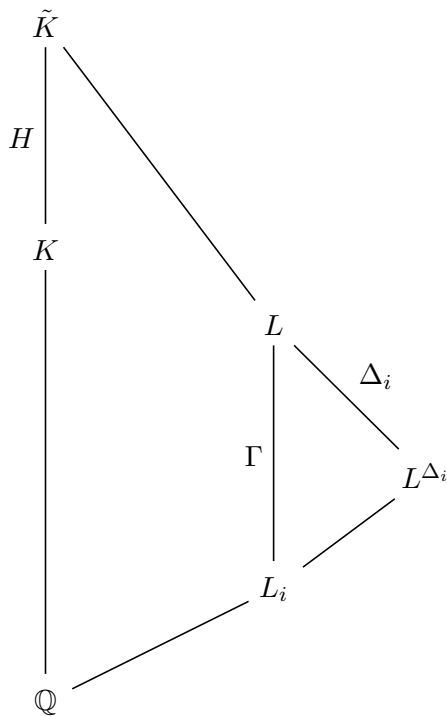
\begin{center}
    \begin{tikzpicture}[->,>=stealth',shorten >=2pt,auto, semithick]
        \node (A) {$\Tilde{K}$};
        \node (B) [below of = A, yshift = -2 cm] {$K$};
        \node (C) [below of = B, yshift = -5 cm] {$\Q$};
        \node (D) [below of = A, yshift = -3cm, xshift = 3cm] {$L$};
        \node (E) [below of = D,yshift = -2.5 cm] {$L_i$};
        \node (F) [below of = D,yshift = -1cm, xshift = 2cm] {$L^{\Delta_i}$};
        \path (A) edge [-] node [left] {$ H $} (B);
        \path (B) edge [-]  (C);
        \path (A) edge [-] node [left] {} (D);
        \path (D) edge [-] node [left] {$\Gamma$} (E);
        \path (D) edge [-] node [above right] {$\Delta_i$} (F);
        \path (F) edge [-] (E);
        \path (E) edge [-]  (C);
    \end{tikzpicture}
    \end{center} 

In this case, if we write the norm relation in $\Q[\Gamma]$ as $1 = \sum_i a_i N_{\Delta_i} b_i$, then we have a generalised norm relation $N_H = \sum_i a_i N_{\Delta_i} (b_i \cdot N_H)$ in $\Q[G]$ that we can use to compute the class group of $K$.

This is equivalent to saying that a subquotient of $G$ admits a classical norm relation, as in proposition \ref{pro:quotient}. In the particular case where $L = \Tilde{K}$, then the generalised norm relation comes from a classical norm relation in a subgroup of $G$, as in remark \ref{req:subgroups}.

We saw already with example \ref{ex: quotient/subgroup} that not all generalised norm relation of a group $G$ comes from a subgroup or a quotient of $G$.  The following algorithm is useful to find examples where generalised norm relation allow to compute class groups more efficiently than classical norm relations in any subgroups or quotients.

\begin{algo} \label{alg:comparaison}

\underline{input:} A group $G$, that is the Galois group of $\Tilde{K}/\Q$

\underline{output:} For all subgroup $H$ of $G$, the maximum of the degree of the subfields of $\Tilde{K}$ one needs to compute the class group of, in order to compute the class group of $\Tilde{K}^H$ using the best classical norm relations in quotients of $G$.

\begin{itemize}
    \item $L_J \leftarrow$ All subgroups of $G$ up to conjugation 
    \item $M \leftarrow$ List of the $\frac{|G|}{|J|}$ for all $J$ in $L_J$. \emph{The entries of $M$ represent the degrees of the $\Tilde{K}^J$. The goal will be to explore all classical norm relations in all quotients of $G$ and update the entries of $M$ to represent the maximum degree of the fields one has to study in order to compute the class group of $\Tilde{K}^J$.}
    \item $M_2 \leftarrow$ An empty list
    \item WHILE $M_2 \neq M$
    \begin{itemize}
        \item $M_2 \leftarrow M$
        \item FOR $i$ from $1$ to $\# L_J$
        \begin{itemize}
            \item $H \leftarrow L_J[i]$
            \item FOR $j$ from $i+1$ to $\# L_J$
            \begin{itemize}
                \item $J \leftarrow L_J[j]$
                \item Check if $H$ is conjugate to a normal subgroup of $J$. If not, go directly to the next $J$.
                \item Look for a classical norm relation in $J/H$ that minimizes the entries of $M$ corresponding to the subgroups involved.
                \item If such a relation is found, update the entries of $M$ accordingly. The entry corresponding to $\Tilde{K}^H$ but also those corresponding to its subfields or all the fields isomorphic to those.
            \end{itemize}
        \end{itemize}
    \end{itemize}
\end{itemize}

\end{algo}

\begin{ex}
Let $G = C_3 \times PSL_3(2)$, and $H = S_3 < G$, suppose we have $\Tilde{K}$ a Galois extension of $\Q$ of Galois group $G$. Then $K = \Tilde{K}^H$ is a field of degree $84$. To compute the class group of $K$, we can verify that there are no classical norm relations in any quotients or subgroups of~$G$ that allows us to recursively reduce the problem to fields of degree less than $84$. However, there exists a generalised norm relation that allows us to reduce the problem to four fields of respective degree $24$, $21$, $8$ and~$3$.

\end{ex}

\begin{req}
    We can do a systematic research by enumerating all groups $G$ up to isomorphism, all subgroups $H <G$, and check every time if there is a generalised norm relation that is more efficient than any classical norm relation in any quotient or subgroups.
    For $|G|< 250$, we find 101 such examples of pairs $(G, H)$.
\end{req}

\begin{req}
    As explained in \cite[theorem 2.11]{norm}, the groups that do not admit classical norm relations are the ones with with a fixed point free unitary representation. We could not find any generalised norm relations in these groups either, except the ones coming from classical norm relations in quotients. We do not know if this is true in general or if counterexamples are simply larger.
\end{req}

In the rest of this section, we will see that if we have an example of a useful generalised norm relation for a finite group $G$, we can build infinitely many other examples, simply by taking the same relation in $C_p \times G$, for any prime $p$ that does not divide $|G|$.

\begin{df}{}
Let $G$ be a group that admits a generalised norm relation with respect to $H < G$ and a set a subgroups $\J = \{ J_1 \cdots J_\ell \}$. we say that the relation is \emph{optimal} if it is the relation that maximizes the quotient $\frac{|J_i|}{|H|}$, where $J_i$ is the smallest group in $\J$.
\end{df}

\begin{req}
With the notations of the previous definition, if $\Tilde{K}/\Q$ is a Galois extension of Galois group $G$, then the quotient $\frac{|J_i|}{|H|}$ is the quotient of the degree of $\Tilde{K}^H$ by the degree of $\Tilde{K}^{J_i}$.
\end{req}

\begin{pro}
Let $G$ be a group that admits a generalised norm relation with respect to $H < G$ and a set a subgroups $\J = \{ J_1 \cdots J_\ell \}$. Suppose this generalised norm relation is optimal. Let $p$ be a prime number that does not divide $|G|$. Then $C_p \times G$ admits an optimal generalised norm relation with respect to $1 \times H$ and $\J_2 = \{ 1 \times J_1, \cdots, 1 \times J_\ell \}$.
\end{pro}

\begin{proof}

Let $G' = C_p \times G$. Let $\rho '$ be an irreducible representation of $G'$. Then $\rho' = \xi \otimes \rho$, whis $\chi$ a character of $C_p$ and $\rho$ an irreducible representation of $G$. 

\begin{lm}
For all subgroup $K'$ of $G'$, either $K'$ is of the form $1 \times K$ with $K < G$, or it is of the form $C_p \times K$ with $K < G$.
\end{lm}

\begin{proof}

Suppose $K'$ contains an element $i \times g \in G' = C_p \times G$ with $i \neq 1$. Let $n$ be the order of $g$ in $G$. Then, since $\gcd(n,p) = 1$,  the subgroup $K'$ contains all the $(k n) i \times 1_G$ with $k$ in $\N$. So $C_p \times 1_G$ is contained in $G'$. So it is easy to check that the projection of $K'$ on $G$ is indeed a subgroup of $G$.

\end{proof}

Let $K$ a subgroup of $G$. Then $(\rho')^{1 \times K} = \rho ^K$ and $(\rho')^{C_p \times K } = \chi ^ {C_p} \otimes \rho ^K$. So $(\rho')^{C_p \times K } \neq 0$ if and only if $\chi$ is trivial and $\rho^K \neq 0$.

Since $G$ admits a generalised norm relation with respect to $H$ and $\J$, then for every irreducible representation $\rho$ of $G$, if $\rho ^H \neq 0$, there exists $J \in \J$ such that $\rho^J \neq 0$. Let $\rho' = \chi \otimes \rho$ be an irreducible representation of $G'$. Then it is easy to check that if $(\rho')^{1 \times H} \neq 0$, there exists  $J \in \J$ such that $(\rho')^{1 \times J} \neq 0$. So $G'$ admits a generalised norm relation with respect to $1 \times H$ and $ \{ 1 \times J_1, \cdots, 1 \times J_\ell \}$.

Now let us prove that this relation is optimal. Suppose $G$ has a better generalised norm relation with respect to $\Tilde{H}' < G'$ and $\{ \Tilde{J_1}',\dots, \Tilde{J_m}' \}$. Let $\Tilde{H}, \Tilde{J_1}, \cdots, \Tilde{J_m} < G$ the projections of $\Tilde{H}'$ and of the $\Tilde{J_i}'$ onto $G$. Then, using the same method as before, it is easy to check that $G$ admits a generalised norm relation with respect to $\Tilde{H}$ and the $\Tilde{J_i}$, and that this norm relation is better than the first one, which is a contradiction.

\end{proof}

\section{Examples} \label{sec:Ex}

\begin{ex}
The group $G = S_5$ admits a generalised norm relation with respect to $H = S_3 < G$ and $\J = \{ A_4, D_{12}, C_5 : C_4 \}$. We can check that this relation does not come from a classical norm relation quotient. There are non conjugate copies of $S_3$ in $S_5$. For $H$ we have to take the one with no fixed points.

If we choose a Galois extension $\Tilde{K}/\Q$ of Galois group $G$, then $K = \Tilde{K}^H$ is of degree $20$, and we can compute its class group inductively, by reducing the problem to three fields of respective degree $10$, $10$ and~$6$.

By choosing $\Tilde{K}$ such that $K$ has a big discriminant, we can obtain examples where the recursive method is more efficient to compute the class group of $K$ than the preexisting methods. For example, consider the polynomial $p(x) = x^5 + 91x^4 + 7x^3 - 11x^2 - x + 1$ and define $\Tilde{K}$ to be the splitting field of $p(x)$. Then $\Tilde{K}$ has Galois group $S_5$, and $K = \Tilde{K}^{S_3}$ is a number field of degree $20$ and of discriminant $2^{28}\cdot 383^{10}\cdot 4723^{10}\cdot 23831^{10} \simeq 6 \cdot 10^{114}$.
On Pari/GP \cite{PARI2}, the function to compute $\Cl(K)$ was not able to finish in three days, whereas with the method of generalised norm relations, implemented also in Pari/GP, we obtained the result in less than nine hours (CPU time). The result is $\Cl(K) = C_4 \times C_2^4$.

\end{ex}

\begin{ex} \label{ex:C7A5}
The group $G = A_5$ admits a a generalised norm relation with respect to $H = C_2 \times C_2 < G$ and $\J = \{ A_4, D_{10} \}$. We can check that this relation does not come from a classical norm relation quotient.

If we choose a Galois extension $\Tilde{K}/\Q$ of Galois group $G$, then $K = \Tilde{K}^H$ is of degree $15$, and we can compute its class group inductively, by reducing the problem to two fields of respective degree $6$ and $5$. However, the method with classical norm relations also applies here, but with that method, the largest field we would need to consider is of degree $12$.

To create a bigger example, since $7 \nmid |A_5|$, we can consider the generalised norm relation of $G' = C_7 \times A_5$ with respect to $H = C_2 \times C_2 < G$ and  $\J = \{ A_4,  D_{10} \}$. That way, we can compute the class group of a field of degree $105$ by reducing the problem to two fields of respective degree $42$ and $35$, whereas with classical norm relations, we would have reduced the problem to a field of degree $84$.

For example, consider the polynomial $f(x) = x^6 - 2x^5 + 3x^4 - 4x^3 + 2x^2 - 2x - 1$. Define $\Tilde{L}$ to be the splitting field of $f(x)$. Then $\Tilde{L}$ has Galois group $A_5$. The splitting field $\Tilde{M}$ of the polynomial $g(x) = x^4 + x^3 + 4x^2 + 20x + 23$ has Galois group $C_7$. Up to isomorphism, there is only one compositum $\Tilde{K}$ of $\Tilde{L}$ and $\Tilde{M}$. What's more, $\Tilde{K}/\Q$ is Galois and its Galois group is $G = C_7 \times A_5$. Denote by $K$ the subfield of $\Tilde{K}$ fixed by $H = C_2 \times C_2$, which is a field of degree $105$ and of discriminant $2^{126} \cdot  29^{90} \cdot  67^{42} \simeq 1.7 \cdot 10^{246}$. 
With the method involving only classical norm relation, we can compute $\Cl(K)$, but we have to compute the class group of some subfields, the largest of which is $F = \Tilde{K}^{C_5}$, of degree $84$ and of discriminant $2^{126} \cdot  29^{72} \cdot  67^{42} \simeq 8 \cdot 10^{219}$. On Pari/GP, the function to compute $\Cl(F)$ was not able to finish in over $5$ months, whereas with our implementation of the method of generalised norm relations, we computed $\Cl(K)$ in about $5$ days (CPU time). The result is $\Cl(K) = 1$.

\end{ex}



\printbibliography[title = References]

\newpage

\end{document}